\documentclass{scrartcl}

\usepackage[english]{babel}
\usepackage[utf8]{inputenc}
\usepackage[T1]{fontenc}
\usepackage{lmodern}

\usepackage{graphicx}
\usepackage{epstopdf}

\usepackage[mode=buildnew]{standalone}
\usepackage{tikz}
\usetikzlibrary{shapes,arrows,positioning,external,calc,math}
\usepackage{pgfplots}
\pgfplotsset{compat=1.16}

\usepackage{url}
\usepackage{verbatim}
\usepackage{multirow, booktabs, adjustbox} 
\usepackage{enumerate}

\usepackage[shortcuts]{extdash}

\usepackage{algorithm}
\usepackage{algorithmic}

\usepackage{amssymb}
\makeatletter
\@ifpackageloaded{amsmath}{}{\usepackage[nosumlimits]{amsmath}}
\makeatother
\usepackage{fixmath} 
\usepackage{xfrac} 
\usepackage[makeroom]{cancel} 
\usepackage{dsfont}
\usepackage{mathrsfs}
\usepackage{mathtools}
\usepackage{centernot}

\newcommand{\norm}[2][]{{#1\|}{#2}{#1\|}}
\newcommand{\inprod}[3][]{{#1\langle}{#2},{#3}{#1\rangle}}

\DeclareMathOperator{\Id}{Id}

\DeclareMathOperator{\dom}{dom}

\DeclareMathOperator{\zer}{zer}
\DeclareMathOperator{\ran}{ran}
\DeclareMathOperator{\gra}{gra}

\newcommand{\R}{\mathbb{R}}

\newcommand{\N}{\mathbb{N}}
\newcommand{\Hprim}{\mathcal{H}}
\newcommand{\Hsec}{\mathcal{K}}
\newcommand{\Hdual}{\mathcal{G}}

\usepackage[numbers,square,sort&compress]{natbib}
\usepackage{hyperref} 
\usepackage[capitalise]{cleveref} 

\crefformat{equation}{(#2#1#3)}
\crefrangeformat{equation}{(#3#1#4) to~(#5#2#6)}
\crefmultiformat{equation}{(#2#1#3)}{ and~(#2#1#3)}{, (#2#1#3)}{ and~(#2#1#3)}

\crefformat{enumi}{(#2#1#3)}
\crefrangeformat{enumi}{(#3#1#4) to~(#5#2#6)}
\crefmultiformat{enumi}{(#2#1#3)}{ and~(#2#1#3)}{, (#2#1#3)}{ and~(#2#1#3)}

\usepackage{amsthm}

\newtheorem{prop}{Proposition}[section]
\numberwithin{prop}{section}

\newtheorem{ass}{Assumption}[section]
\numberwithin{ass}{section}

\newtheorem{lem}{Lemma}[section]
\numberwithin{lem}{section}

\newtheorem{thm}{Theorem}[section]
\numberwithin{thm}{section}

\newtheorem{cor}{Corollary}[section]
\numberwithin{cor}{section}

\numberwithin{defin}{section}

\newtheorem{rem}{Remark}[section]
\numberwithin{rem}{section}

\makeatletter
\@ifpackageloaded{cleveref}{
	\crefname{prop}{Proposition}{Propositions}
	\crefname{ass}{Assumption}{Assumptions}
	\crefname{lem}{Lemma}{Lemmas}
	\crefname{thm}{Theorem}{Theorems}
	\crefname{cor}{Corollary}{Corollaries}
	\crefname{defin}{Definition}{Definitions}
	\crefname{rem}{Remark}{Remarks}
}{}
\makeatother

\newcommand{\remove}[1]{}

\title{Nonlinear Forward-Backward Splitting with Momentum Correction}
\author{
	Martin Morin
	\thanks{Department of Automatic Control, Lund University, (%
		\url{martin.morin@control.lth.se},
		\url{sebastian.banert@control.lth.se},
		\url{pontus.giselsson@control.lth.se}%
		)
	}
	\and
	Sebastian Banert\footnotemark[1]
	\and
	Pontus Giselsson\footnotemark[1]
}
\date{}

\begin{document}

\maketitle

\begin{abstract}
	\noindent
	The nonlinear, or warped, resolvent recently explored by Giselsson and B\`ui\--Combettes has been used to model a large set of existing and new monotone inclusion algorithms.
To establish convergent algorithms based on these resolvents, corrective projection steps are utilized in both works.
We present a different way of ensuring convergence by means of a nonlinear momentum term, which in many cases leads to cheaper per-iteration cost.
The expressiveness of our method is demonstrated by deriving a wide range of special cases.
These cases cover and expand on the forward-reflected-backward method of Malitsky\--Tam, the primal-dual methods of V\~u\--Condat and Chambolle\--Pock, and the forward-reflected-Douglas\--Rachford method of Ryu\--V\~u.
A new primal-dual method that uses an extra resolvent step is also presented as well as a general approach for adding momentum to any special case of our nonlinear forward-backward method, in particular all the algorithms listed above.

\end{abstract}

\section{Introduction}\label{sec:intro}

Given a real Hilbert space $\Hprim$, we consider the problem of finding a zero $x\in \Hprim$ of the sum of a maximally monotone operator $A\colon\Hprim\to 2^\Hprim$ and a cocoercive operator $C\colon\Hprim\to\Hprim$, i.e.,
\begin{equation}\label{eq:prob}
	0 \in Ax + Cx.
\end{equation}
If the resolvent $(\Id + A)^{-1}$ of $A$ is easily computable, this problem can be solved with the forward-backward splitting method \cite{goldsteinConvexProgrammingHilbert1964,levitinConstrainedMinimizationMethods1966}.
Since this might not be the case, great effort has been devoted to constructing other splitting methods that can exploit any additional structure of $A$, sometimes further assuming $C=0$ \cite{lionsSplittingAlgorithmsSum1979, davisThreeOperatorSplittingScheme2017,combettesPrimalDualSplittingAlgorithm2012,botSolvingSystemsMonotone2013,raguetGeneralizedForwardBackwardSplitting2013,combettesAsynchronousBlockiterativePrimaldual2018,combettesSolvingCompositeFixed2020}.
This work presents an alternative approach for analyzing and constructing such splitting methods by formulating them as different instances of a forward-backward method with a nonlinear resolvent $(M+A)^{-1}\circ M$ where $M\colon\Hprim\to\Hprim$ is a (potentially) nonlinear kernel.

Nonlinear resolvents---or warped resolvents in the terminology of \cite{buiWarpedProximalIterations2020}---were recently explored in \cite{giselssonNonlinearForwardBackwardSplitting2021,buiWarpedProximalIterations2020} with precursors available in \cite{kassay1985proximal,konnovCombinedRelaxationMethods2006}.
These works are preceded by, or developed in parallel with, several other generalizations to the concept of a resolvent.
Using a resolvent with a strongly positive self-adjoint bounded linear kernel $P$ in the standard forward-backward method has long been known to converge.
In fact, it is simply forward-backward splitting applied to the scaled problem $0 \in P^{-1} Ax + P^{-1} Cx$, which is a monotone inclusion problem in the Hilbert space given by the inner product $\inprod{P(\cdot)}{\cdot}$.
The conditions on the kernel have been further relaxed in \cite{latafatAsymmetricForwardBackward2017}, which allows for non-self-adjoint linear kernels.
In the multiple works on Bregman-distance based resolvents, for instance \cite{bregmanRelaxationMethodFinding1967,bauschkeBregmanMonotoneOptimization2003a,buiBregmanForwardBackwardOperator2021,ecksteinNonlinearProximalPoint1993,bauschkeJointMinimizationAlternating2006,burachikInexactProximalPoint2010,bauschkeRegularizingBregmanMoreau2018}, the linearity condition is dropped altogether by allowing the kernel to be the gradient of some differentiable convex function.
These relaxations allow the resolvent to be adapted to a particular problem, either to improve the speed of convergence or to make an otherwise intractable resolvent evaluation tractable.
However, this extra freedom may come at a cost.
The algorithms of \cite{latafatAsymmetricForwardBackward2017,giselssonNonlinearForwardBackwardSplitting2021,buiWarpedProximalIterations2020,konnovCombinedRelaxationMethods2006} all need an extra corrective projection step to ensure that any nonlinearities and asymmetries of the kernel do not prevent convergence.
The primary contribution of this paper is a different approach for correcting the update, removing the need to perform a potentially expensive projection.
Convergence is instead ensured with a corrective momentum term that reuses information from previous iterations, making it possible to achieve lower per-iteration costs.

The strength of nonlinear resolvents lies in their substantial modeling power which allows for a unified view of a large set of algorithms.
Both \cite{giselssonNonlinearForwardBackwardSplitting2021,buiWarpedProximalIterations2020} present numerous algorithms that can be interpreted as forward-backward methods with nonlinear resolvents.
Our new nonlinear forward-backward method further expands on these modeling capabilities and the second half of this paper is dedicated to deriving both new and existing algorithms as special cases.

Among already existing methods, we show that the forward-(half)-reflected-backward method in \cite{malitskyForwardBackwardSplittingMethod2020} is a special case of our method and highlight its connection to the similar forward-backward-(half)-forward method \cite{tsengModifiedForwardBackwardSplitting2000,briceno-ariasForwardBackwardHalfForwardAlgorithm2018} via the nonlinear resolvent.
We present two new four-operator primal-dual splitting methods, the first of which has, among others, V\~u\--Condat \cite{vuSplittingAlgorithmDual2013,condatPrimalDualSplitting2013} and Chambolle\--Pock \cite{chambolleFirstOrderPrimalDualAlgorithm2011} as special cases.
V\~u\--Condat and Chambolle\--Pock have been shown to be ordinary forward-backward methods \cite{heConvergencePrimalDualHybrid2014} and to have Douglas\--Rachford splitting \cite{lionsSplittingAlgorithmsSum1979} as a special case.%
\footnote{
	In order to formulate the standard Douglas\--Rachford as a forward-backward method, singular resolvent kernels needs to be allowed.
	The analysis of this paper will not allow for this but can be modified to do so.
	}
Our first primal-dual method is an expansion of this to the nonlinear resolvent setting, giving us the forward-reflected-Douglas\--Rachford method of \cite{ryuFindingForwardDouglasRachfordForward2020} and the novel forward-half-reflected-Douglas\--Rachford method as special cases.
Our second primal-dual method solves the same problem as the first one but utilizes three resolvent steps, two of which are of the same operator.
This method is, as far as we know, completely novel.

Different kinds of momentum have long been used to accelerate the convergence of first-order methods \cite{polyakMethodsSpeedingConvergence1964,alvarezInertialProximalMethod2001,moudafiConvergenceSplittingInertial2003,botInertialDouglasRachford2015,lorenzInertialForwardBackwardAlgorithm2015,beckFastIterativeShrinkageThresholding2009,attouchConvergenceRatesInertial2018,botInertialForwardbackwardforwardPrimaldual2016} and, due to the use of a momentum-like correction term, our nonlinear forward-backward method naturally lend itself to modeling momentum methods.
Momentum can be incorporated directly into the design of a special case of our main algorithm but we also present an approach to add momentum to any special case, regardless of whether it initially was designed with momentum or not.
The approach is demonstrated on the forward-half-reflected-backward method of \cite{malitskyForwardBackwardSplittingMethod2020}, which gives a novel momentum algorithm that extends the relaxed momentum algorithm in \cite{malitskyForwardBackwardSplittingMethod2020} to include a cocoercive term.
Our convergence conditions compare favorably to previous work with a larger range of possible choices of the momentum parameter, even in the more restrictive special case of ordinary forward-backward splitting with momentum.

\subsection{Outline}
We start by presenting basic notation, preliminary results, and define some operator properties.
The proposed nonlinear forward-backward algorithm, along with all necessary assumptions on both the problem \cref{eq:prob} and the different design parameters, is presented in \cref{sec:prob-alg}.
\Cref{sec:conv} contains the main convergence proof.

In the remainder of the paper, we present and discuss new or already existing special cases of our nonlinear forward-backward method.
\Cref{sec:add-momentum} presents a way of adding momentum to any special case of our main algorithm.
\Cref{sec:fhrb} derives the forward-half-reflected-backward method of \cite{malitskyForwardBackwardSplittingMethod2020} as a special case and uses the previously presented approach to add momentum to it.
Two new primal-dual methods are derived in \cref{sec:primdual}.
\Cref{sec:primdual-tri} contains an algorithm that expands on the methods of V\~u\--Condat and Chambolle\--Pock as well as the forward-reflected-Douglas\--Rachford of \cite{ryuFindingForwardDouglasRachfordForward2020}.
In \Cref{sec:primdual-res} a, to the authors' knowledge, completely new primal-dual method that uses one additional resolvent evaluation per iteration is derived.
We end the paper with a brief conclusion.

\subsection{Notation and Preliminaries}
Let $\R$ be the set of real numbers, $\N = \{0,1,\dots\}$ be the set of natural numbers, $\N_+ = \{1,2,\dots\}$ be the set of non-zero natural numbers, and let $\Hprim$ be a real Hilbert space.
The set $\mathcal{P}(\Hprim)$ is the set of bounded linear operators $S\colon\Hprim \to \Hprim$ that are self-adjoint and strongly positive, i.e., there exists $m > 0$ such that
\[
	\inprod{Sx}{x} \geq m\norm{x}^2,\quad \forall x\in\Hprim.
\]
If $S \in \mathcal{P}(\Hprim)$, then $S$ is invertible and $S^{-1} \in \mathcal{P}(\Hprim)$.

For the remainder of this section, we let $S\in\mathcal{P}(\Hprim)$.
The scaled inner product is defined as $\inprod{\cdot}{\cdot}_S = \inprod{S(\cdot)}{\cdot}$ and the scaled norm as $\norm{\cdot}_{S} = \sqrt{\inprod{\cdot}{\cdot}_S}$.
The unscaled and scaled norms are equivalent, i.e., there exist $M,m>0$ such that $M\norm{x} \geq \norm{x}_S \geq m\norm{x}$ for all $x\in\Hprim$.
For all $a,b,c,d\in\Hprim$, we have the identity
\begin{equation}\label{eq:inprod-diff-eq}
\begin{aligned}
	2\inprod{a-b}{d-c}_S
	&= \|a-c\|_S^2 - \|b-c\|_S^2 - \|a-d\|_S^2 + \|b-d\|_S^2.
\end{aligned}
\end{equation}

A set-valued operator $A\colon\Hprim\to 2^\Hprim$ is \emph{monotone} if
\[
	\inprod{u-v}{x-y} \geq 0,\quad \forall (x,u),(y,v) \in \gra A
\]
where
$
\gra A = \{(x,u) \mid u \in Ax\}
$
is the graph of $A$.
An operator $A$ is \emph{maximally monotone} if it is monotone and its graph is not a proper subset of the graph of another monotone operator.

For $\mu > 0$, a maximally monotone operator $A\colon\Hprim \to 2^\Hprim$ is \emph{$\mu$-strongly monotone w.r.t.\ $S$} if
\[
	\inprod{u-v}{x-y} \geq \mu\norm{x-y}_S^2,\quad \forall u\in Ax, \forall v\in Ay, \forall x,y\in\Hprim.
\]
This definition is equivalent to ordinary $\mu$-strong monotonicity of $S^{-1}\circ A$ in the Hilbert space given by the scaled inner product $\inprod{\cdot}{\cdot}_S$.
The analogous equivalences hold for the two following definitions as well.
For $L \geq 0$, an operator $B\colon\Hprim \to\Hprim$ is \emph{$L$-Lipschitz continuous w.r.t.\ $S$} if
\[
	\norm{Bx - By}_{S^{-1}} \leq L \norm{x-y}_S,\quad \forall x,y\in\Hprim.
\]
For $\ell > 0$, an operator $C\colon\Hprim \to\Hprim$ is \emph{$\ell^{-1}$-cocoercive w.r.t.\ $S$} if
\[
	\inprod{Cx-Cy}{x-y} \geq \ell^{-1}\norm{Cx-Cy}_{S^{-1}}^2,\quad \forall x,y\in\Hprim.
\]
An $\ell^{-1}$-cocoercive operator w.r.t.\ $S$ is $\ell$-Lipschitz continuous w.r.t.\ $S$.
For all operator properties, if no scaling $S$ is explicitly stated, we mean $S=\Id$.

Let $C$ be an $\ell^{-1}$-cocoercive operator w.r.t.\ $S$. Then the following three-point inequality holds:
\begin{equation}
	\label{eq:coco-three-point}
	\inprod{Cx-Cy}{z-y} \geq -\tfrac{\ell}{4}\norm{z-x}_{S}^2,\quad \forall x,y,z\in\Hprim.
\end{equation}
This is shown by inserting $x-x$ in the inner product on the left-hand side and using cocoercivity and Young's inequality,
\begin{align*}
	\inprod{Cx - Cy}{z - y}
	&= \inprod{Cx - Cy}{z-x} + \inprod{Cx - Cy}{x-y} \\
	&\geq \inprod{Cx - Cy}{z - x} + \ell^{-1}\norm{Cx - Cy}_{S^{-1}}^2 \\
	&= \inprod{S^{-\frac{1}{2}}(Cx - Cy)}{S^{\frac{1}{2}}(z - x)} + \ell^{-1}\norm{Cx - Cy}_{S^{-1}}^2 \\
	&\geq -\tfrac{\epsilon}{2}\norm{Cx - Cy}_{S^{-1}}^2 - \tfrac{1}{2\epsilon}\norm{z-x}_{S}^2 + \ell^{-1}\norm{Cx - Cy}_{S^{-1}}^2
\end{align*}
where $\epsilon > 0$.
Selecting $\epsilon = 2\ell^{-1}$ yields the desired inequality \eqref{eq:coco-three-point}. If $C=0$ or is constant, \eqref{eq:coco-three-point} holds with $\ell=0$.

\section{Problem and Algorithm}\label{sec:prob-alg}
Apart form the general problem structure of \cref{eq:prob}, we further assume that the operators satisfy the following standard assumptions.
\begin{ass}\label{ass:prob-assump}
	The operators of \cref{eq:prob} satisfy:
	\begin{enumerate}[(i)]
		\itemsep0em
		\item $A\colon \Hprim \to 2^\Hprim$ is maximally monotone.
		\item $C\colon \Hprim \to \Hprim$ is $\ell^{-1}$-cocoercive w.r.t.\ $S$, where $S\in\mathcal{P}(\Hprim)$.
		\item $\zer (A+C) \neq \emptyset$.
	\end{enumerate}
	If $C = 0$, we set $\ell = 0$.
\end{ass}
Since $\dom C = \Hprim$, the sum $A+C$ is maximally monotone and the problem could be reformulated as finding a zero of the single maximally monotone operator $A+C$.
However, as in ordinary forward-backward splitting, separating the problem into a maximally monotone and a cocoercive term and utilizing this structure will prove beneficial.
The fact that we assume cocoercivity w.r.t.\ $S$ entails no real restriction on the problem since the scaled norm $\norm{\cdot}_S$ is equivalent to $\norm{\cdot}$.
A cocoercive operator w.r.t.\ $S$ is therefore also cocoercive w.r.t.\ all other $\hat{S}\in\mathcal{P}(\Hprim)$ and vice versa, but with different cocoercivity constants.

The cocoercivity scaling $S$ is utilized directly in our algorithm.
In the simplest setting, $S$ acts as a form of preconditioning used to better adapt the algorithm to the specific geometry of the problem.
It can also be used as a more general design parameter with different choices of $S$ yielding different instances of our algorithm, see the primal-dual methods in \cref{sec:primdual} for examples of this.
Along with the scaling $S$, the algorithm has two additional iteration-dependent design parameters, a nonlinear kernel $M_k\colon\Hprim \to \Hprim$ and a positive momentum parameter $\gamma_k > 0$:
\begin{algorithm}[H]
	\caption{Nonlinear Forward-Backward with Momentum Correction}
	\label{alg:nofob_mom}
	Consider problem \cref{eq:prob} and let $S$ be such that \cref{ass:prob-assump} is satisfied.
	With $x_0, u_0 \in \Hprim$, for all $k\in\N$ iteratively perform
	\begin{align*}
		x_{k+1} &=(M_k + A)^{-1}(M_kx_k - Cx_k + \gamma_k^{-1}u_k), \\
		u_{k+1} &= (\gamma_kM_{k} - S)x_{k+1} - (\gamma_kM_{k} - S)x_{k},
	\end{align*}
	where $M_k\colon\Hprim\to\Hprim$ and $\gamma_k > 0$.
\end{algorithm}

Compared to \cite{giselssonNonlinearForwardBackwardSplitting2021,buiWarpedProximalIterations2020}, the elements of the sequence $(x_k)_{k\in\N}$ are given directly by a nonlinear forward-backward step and do not need an extra projection step.
Convergence is instead ensured by the addition of the corrective term $u_k$ to the forward step.
The main benefit of this approach is in how the corrective term $u_k$ is computed.
Both \cref{alg:nofob_mom} and the corresponding algorithm with projection correction \cite[Algorithm~3.1]{giselssonNonlinearForwardBackwardSplitting2021} will in general need to evaluate $M_k$ at two points.
For \cref{alg:nofob_mom}, the two points are $x_k$ and $x_{k+1}$ but this means that $M_k$ and $M_{k+1}$ are evaluated at the same point, i.e., $x_{k+1}$.
The cost of one of these evaluations can then be reduced if $M_{k}$ and $M_{k+1}$ are similar, for instance if $M_{k+1}x_{k+1}$ is a scalar multiplication of $M_kx_{k+1}$.
In order for \cite[Algorithm 3.1]{giselssonNonlinearForwardBackwardSplitting2021} to also evaluate $M_k$ at $x_k$ and $x_{k+1}$, it is required that all $M_k=\alpha_k^{-1}S$ with $S\in\mathcal{P}(\Hprim)$ and $\alpha_k > 0$ for all $k\in\N$.
The only instance of \cite[Algorithm 3.1]{giselssonNonlinearForwardBackwardSplitting2021} that satisfies this condition is ordinary forward-backward splitting in the scaled metric given by $\norm{\cdot}_{S}$.
This is in contrast to our work where all but one---\cref{alg:primdual-res}---of the special cases we cover have kernels that allow this reduction in computational cost.

The more similar $M_k$ and $\gamma_k^{-1}S$ are in \cref{alg:nofob_mom}, the more similar the nonlinear resolvent is to an ordinary scaled resolvent $(\gamma_k^{-1}S + A)^{-1}\circ\gamma_k^{-1}S$ and the smaller the corrective term $u_{k+1}$ will be.
No correction, i.e., $u_{k+1}=0$, is applied when $M_k=\gamma_k^{-1}S$ and \cref{alg:nofob_mom} then reduces to ordinary forward-backward splitting.
We quantify the difference between $M_k$ and $\gamma_k^{-1}S$ in the following assumption on the design parameters of \cref{alg:nofob_mom}.
\begin{ass}\label{ass:alg-assump}
	Assume that:
	\begin{enumerate}[(i)]
		\itemsep0em
		\item\label{ass:item:alg-assump-gamma} The sequence $(\gamma_k)_{k\in\N}$ is positively lower bounded, i.e., for each $k\in\N$, $\gamma_k \geq \gamma$ for some $\gamma > 0$.
		\item\label{ass:item:alg-assump-M} For each $k\in \mathbb{N}$, the nonlinear kernel $M_k\colon \Hprim \to \Hprim$ is such that 
   $\gamma_kM_k-S$ is $L_k$-Lipschitz continuous w.r.t. $S$, for some $L_k\geq 0$.
	\end{enumerate}
\end{ass}

These assumptions will form the basis of our convergence analysis.
First, we will use them to infer a few useful properties of the nonlinear kernel $M_k$.
\begin{prop}\label{prop:kernel-properties}
	Let \cref{ass:alg-assump} hold with $L_k \in[0, 1)$ for all $k\in\N$.
	Then $M_k$ is $2\gamma^{-1}$-Lipschitz continuous w.r.t.\ $S$, maximally monotone, and strongly monotone w.r.t.\ $S$ for all $k\in\N$.
	\begin{proof}
		The kernel $M_k$ satisfies $M_k = \gamma_k^{-1}(\gamma_kM_k - S) + \gamma_k^{-1}S$ and therefore is it $\gamma_k^{-1}(1+L_k)$-Lipschitz continuous w.r.t.\ $S$.
		Since $L_k < 1$ and $\gamma_k \geq \gamma$, the Lipschitz continuity claim is proven.
  Let $\rho_k=\tfrac{1-L_k}{\gamma_k}$. Then $L_k$-Lipschitz continuity of $\gamma_kM_k - S$ gives
  \begin{align*}
			L_k^2\norm{x-y}_S^2
			&\geq \norm{(\gamma_kM_k - S)x - (\gamma_kM_k - S)y}_{S^{-1}}^2 \\
			&= \norm{\gamma_kM_kx - \gamma_kM_ky-\rho_k\gamma_k S(x-y)}_{S^{-1}}^2 + \norm{(1-\rho_k\gamma_k)S(x-y)}_{S^{-1}}^2 \\
			&\quad - 2\gamma_k\inprod{M_kx-M_ky-\rho_k S(x-y)}{(1-\rho_k\gamma_k)x-y} \\
			&\geq L_k^2\norm{x-y}_{S}^2 - 2\gamma_kL_k\inprod{M_kx-M_ky-\rho_k S(x-y)}{x-y},
		\end{align*}
  where we have used $(1-\rho_k\gamma_k)=L_k$. Since $\gamma_k>0$ and $L_k>0$, we conclude that $M_k-\rho_k S$ is monotone and that $M_k$ is $\rho_k$-strongly monotone w.r.t. $S$.
		Maximality of $M_k$ follows from its continuity and monotonicity \cite[Corollary 20.28]{bauschkeConvexAnalysisMonotone2017}.
	\end{proof}
\end{prop}

\section{Convergence}\label{sec:conv}

The convergence of \cref{alg:nofob_mom} will be established by the convergence of a quantity $\mathcal V_k$, defined in \cref{lem:lyapunov}.
The quantity $\mathcal V_k$ consists of the distance from the corrected iterate $x_k + S^{-1}u_k$ to an arbitrary solution (measured in the scaled norm $\norm{\cdot}_S$) and a residual term.
\Cref{thm:main-conv} will then establish the main convergence result.
Before that, we show that the algorithm generates a well-defined infinite sequence.

\begin{prop}\label{prop:alg-well-posed}
	Let \cref{ass:prob-assump,ass:alg-assump} hold with $L_k \in[0, 1)$ for all $k\in\N$.
	Then \cref{alg:nofob_mom} generates infinite sequences $(x_k)_{k\in\N}$ and $(u_k)_{k\in\N}$ uniquely determined by $x_0$ and $u_0$.
	\begin{proof}
		Since $S$, $C$, and $M_k$ are single-valued, it suffices to show that $(M_k + A)^{-1}$ is also single-valued and has full domain.
		By \cref{prop:kernel-properties}, the kernel $M_k$ is maximally monotone and strongly monotone w.r.t.\ $S$, which implies maximal monotonicity and strong monotonicity w.r.t.\ $\Id$ as well.
		The kernel has full domain, $\dom M_k = \Hprim$, so the sum $M_k + A$ is maximally monotone and strongly monotone with $\ran (M_k+A) = \Hprim$ and hence $\dom (M_k + A)^{-1} = \Hprim$ \cite[Corollary 25.28]{bauschkeConvexAnalysisMonotone2017}.
		Since $M_k + A$ is strongly monotone, $(M_k+A)^{-1}$ is cocoercive and hence Lipschitz continuous and single-valued \cite[Example 22.7]{bauschkeConvexAnalysisMonotone2017}.
	\end{proof}
\end{prop}

\begin{lem}\label{lem:lyapunov}
	Let $z \in \zer (A + C)$ and let \cref{ass:prob-assump,ass:alg-assump} hold with $L_k < 1$ for all $k\in\N$.
	Then \cref{alg:nofob_mom} satisfies
	\begin{equation}\label{eq:lyapunov-inequality}
		(1 - L_{k-1} - L_k - \tfrac{\gamma_k\ell}{2} )\norm{ x_{k+1} - x_k}_{S}^2 \leq \mathcal{V}_{k}  - \mathcal{V}_{k+1}
	\end{equation}
	for all $k\in\N_+$ where
	\begin{align*}
		\mathcal{V}_k = \norm{x_k + S^{-1}u_k - z }_{S}^2 + (1-L_{k-1})L_{k-1}\norm{ x_{k} - x_{k-1} }_{S}^2.
	\end{align*}

	\begin{proof}
		By \cref{prop:alg-well-posed} we have that sequences $(x_k)_{k\in\N}$ and $(u_k)_{k\in\N}$ are well-defined, which implies that all quantities of the lemma are well-defined.
		Let $k \in \N_+$ be arbitrary.
		From \cref{alg:nofob_mom} we know that
		\begin{align*}
			&x_{k+1} = (M_k+A)^{-1}(M_kx_k-Cx_k+\gamma_k^{-1}u_k).
		\end{align*}
		Using the definition of $(M_k + A)^{-1}$, multiplying with $\gamma_k$ and rearranging yields
		\begin{align*}
			Sx_{k} - Sx_{k+1} + u_{k} - u_{k+1} - \gamma_{k}Cx_k \in \gamma_{k}A x_{k+1}.
		\end{align*}
		Since $z \in \zer (A+C$), we have $-Cz \in Az$.
		Using monotonicity of $\gamma_k A$ and multiplying by $2$ gives
		\begin{align*}
			0
			&\leq 2\inprod{Sx_{k} - Sx_{k+1} + u_{k} - u_{k+1} - \gamma_{k}Cx_k + \gamma_{k}Cz}{x_{k+1} - z} \\
			&= 2\inprod{Sx_{k} + u_{k} - (Sx_{k+1} + u_{k+1})}{x_{k+1} - z} - 2\gamma_k\inprod{Cx_k - Cz}{x_{k+1} - z}.
		\end{align*}
Applying \cref{eq:coco-three-point} on the last term gives
		\begin{align*}
			0 \leq 2\inprod{\xi_{k} - \xi_{k+1}}{x_{k+1} - z}_{S} + \tfrac{\gamma_k\ell}{2}\norm{x_{k+1} - x_k}_{S}^2
		\end{align*}
		where we have set $\xi_k \coloneqq x_k + S^{-1}u_k$.
		Applying \eqref{eq:inprod-diff-eq} to the inner product with $a = \xi_{k}$, $b = \xi_{k+1}$, $c = z$, $d = x_{k+1}$
		yields
		\begin{equation}\label{eq:proof-midstep}
			\begin{aligned}
				0
				&\leq \norm{\xi_k - z}_{S}^2 - \norm{\xi_{k+1} - z}_{S}^2 + \tfrac{\gamma_k\ell}{2}\norm{x_{k+1} - x_k}_{S}^2 \\
				&\quad- \norm{\xi_{k} - x_{k+1}}_{S}^2 + \norm{\xi_{k+1} - x_{k+1}}_{S}^2 \\
				&= \norm{\xi_k - z}_{S}^2 - \norm{\xi_{k+1} - z}_{S}^2 + \tfrac{\gamma_k\ell}{2}\norm{x_{k+1} - x_k}_{S}^2 \\
				&\quad- \norm{S^{-1}u_k - (x_{k+1} - x_{k}) }_{S}^2 + \norm{u_{k+1}}_{S^{-1}}^2.
			\end{aligned}
		\end{equation}
		We can expand the second to last norm, assume $L_{k-1} > 0$ and use Young's inequality to get
		\begin{align*}
			\norm{S^{-1}u_k - (x_{k+1} - x_{k}) }_{S}^2
			&= \norm{u_k}_{S^{-1}}^2 - 2\inprod{u_k}{x_{k+1}-x_k} + \norm{x_{k+1}-x_{k}}_{S}^2 \\
			&\geq - (L_{k-1}^{-1} - 1)\norm{u_k}_{S^{-1}}^2 + (1-L_{k-1})\norm{x_{k+1}-x_{k}}_{S}^2.
		\end{align*}
		By definition we have $u_k = (\gamma_{k-1}M_{k-1}-S)x_k - (\gamma_{k-1}M_{k-1}-S)x_{k-1}$ which yields
		\begin{align*}
			&\norm{S^{-1}u_k + (x_{k} - x_{k+1}) }_{S}^2 \\
			&\quad\geq -(1-L_{k-1})L_{k-1}\norm{x_{k} - x_{k-1}}_{S}^2 + (1-L_{k-1})\norm{x_{k+1}-x_{k}}_{S}^2
		\end{align*}
		since $\gamma_{k-1}M_{k-1}-S$ is $L_{k-1}$-Lipschitz continuous w.r.t.\ $S$ with $L_{k-1} < 1$.
		We also note that this inequality holds when $L_{k-1} = 0$ since $u_k = 0$ in that case.

		Inserting this back into \cref{eq:proof-midstep} and using Lipschitz continuity of $\gamma_kM_k-S$ on the last term yield
		\begin{align*}
			0
			&\leq \norm{\xi_k - z}_{S}^2 - \norm{\xi_{k+1} - z}_{S}^2 + \tfrac{\gamma_k\ell}{2}\norm{x_{k+1} - x_k}_{S}^2 \\
			&\quad+ (1-L_{k-1})L_{k-1}\norm{x_{k} - x_{k-1}}_{S}^2 - (1-L_{k-1})\norm{x_{k+1}-x_{k}}_{S}^2 \\
			&\quad+ L_{k}^2\norm{x_{k+1} - x_k}_{S}^2 \\
			&= \norm{\xi_k - z}_{S}^2 + (1-L_{k-1})L_{k-1}\norm{x_{k} - x_{k-1}}_{S}^2 \\
			&\quad - \norm{\xi_{k+1} - z}_{S}^2 - (1-L_{k})L_{k}\norm{x_{k+1} - x_{k}}_{S}^2 \\
			&\quad - (1-L_{k-1} - L_{k} - \tfrac{\gamma_k\ell}{2})\norm{x_{k+1}-x_{k}}_{S}^2.
		\end{align*}
		Rearranging this expression gives the inequality of the lemma.
	\end{proof}
\end{lem}

\begin{thm}\label{thm:main-conv}
	Let \cref{ass:prob-assump,ass:alg-assump} hold.
	If there exists an $\epsilon > 0$ such that
	\begin{equation}\label{eq:thm:conv-cond}
		1 - L_{k-1} - L_k - \tfrac{\gamma_k \ell}{2} \geq \epsilon
	\end{equation}
	for all $k\in\N_+$, then \cref{alg:nofob_mom} satisfies the following as $k\to\infty$:
	\begin{enumerate}[(i)]
		\itemsep0em
		\item\label{thm:itm:stepconv} $x_{k+1} - x_k \to 0$,
		\item\label{thm:itm:momconv} $u_k \to 0$,
		\item\label{thm:itm:resconv} $(A+C)x_{k+1} \ni M_kx_k - M_kx_{k+1} + \gamma_k^{-1}u_{k} + Cx_{k+1} -  Cx_{k} \to 0$,
		\item\label{thm:itm:weakconv} $x_{k} \rightharpoonup x^\star$ for some $x^\star \in \zer (A+C)$.
	\end{enumerate}
	\begin{proof}
		Let $z\in\zer (A+C)$.
		Applying \cref{lem:lyapunov} and adding the inequality \cref{eq:lyapunov-inequality} for $k=1,\dots,n$ yields
		\begin{align*}
			\sum_{k=1}^n (1 - L_{k-1} - L_k - \tfrac{\gamma_k\ell}{2} )\norm{x_{k+1} - x_k}_{S}^2 \leq \mathcal{V}_1 - \mathcal{V}_{n+1} < \mathcal{V}_1 < \infty.
		\end{align*}
		The second to last inequality holds since $0 \leq L_k < 1$ for all $k\in\N$ by the assumptions and the condition \cref{eq:thm:conv-cond} of the theorem and therefore is $\mathcal{V}_{n+1}$ nonnegative.
		Item \cref{thm:itm:stepconv} follows from letting $n\to\infty$ since $(1 - L_{k-1} - L_k -\tfrac{\gamma_k\ell}{2}) \geq \epsilon > 0$ for all $k\in\N_+$ by the condition of the theorem.
		Item \cref{thm:itm:momconv} follows from \cref{thm:itm:stepconv}, the definition of $u_k$, and from the $L_k$-Lipschitz continuity of $\gamma_kM_k-S$ where $L_k < 1$ for all $k\in\N$.

		Let $k \in \N$.
		For \cref{thm:itm:resconv}, we first note from the nonlinear forward-backward step in \cref{alg:nofob_mom} that
		\begin{align*}
			Ax_{k+1} \ni M_kx_k - M_kx_{k+1} + \gamma_k^{-1}u_{k} - Cx_{k},
		\end{align*}
		which, by adding $Cx_{k+1}$ to both sides, gives
		\begin{align*}
			(A+C)x_{k+1} \ni M_kx_k - M_kx_{k+1} + \gamma_k^{-1}u_{k} + Cx_{k+1} - Cx_{k}.
		\end{align*}
		The result then follows from \cref{thm:itm:stepconv,thm:itm:momconv} since for all $k\in\N$, $\gamma_k > \gamma$ and $M_k$ and $C$ are Lipschitz continuous w.r.t.\ $S$ with constants $2\gamma^{-1}$ and $\ell$ respectively, see \cref{prop:kernel-properties,ass:prob-assump}.

		Since $A+C$ is maximally monotone, \cref{thm:itm:resconv} implies that all weak sequential cluster points of $(x_k)_{k\in\N}$ belong to $\zer (A+C)$ due to weak-strong sequential closedness of graphs of maximal monotone operators \cite[Proposition 20.38]{bauschkeConvexAnalysisMonotone2017}.
		To show the weak convergence result in \cref{thm:itm:weakconv}, in view of \cite[Lemma 2.47]{bauschkeConvexAnalysisMonotone2017}, it is enough to show that $(\norm{x_k - z}_S)_{k\in\N}$ converges for all $z\in\zer (A+C)$. The proof of \cite[Lemma 2.47]{bauschkeConvexAnalysisMonotone2017} actually only covers the case when $(\norm{x_k-z})_{k\in\N}$ converges but the generalization is straightforward.

		For any $z\in\zer (A+C)$, \cref{lem:lyapunov} and the condition $(1 - L_{k-1} - L_k -\tfrac{\gamma_k\ell}{2}) \geq \epsilon > 0$ give that $(\mathcal{V}_k)_{k\in\N_+}$ is a nonincreasing nonnegative sequence which therefore converges, say, $\mathcal{V}_k \to \nu$.
		This convergence implies
		\begin{align*}
			\norm{x_k + S^{-1}u_k - z}_S^2 = \mathcal{V}_k - (1-L_{k-1})L_{k-1}\norm{x_{k} - x_{k-1}}_S^2 \to \nu
		\end{align*}
		due to \cref{thm:itm:stepconv} and $0 \leq L_{k-1} < 1$.
		The sequence $\{x_k + S^{-1}u_k - z\}_{k\in\N}$ is then bounded, which, together with \cref{thm:itm:momconv}, yields
		\begin{align*}
			\norm{x_k - z}_S^2
			&= \norm{(x_k + S^{-1}u_k - z) - S^{-1}u_k}_S^2 \\
			&= \norm{x_k + S^{-1}u_k - z}_S^2 + \norm{u_k}_{S^{-1}}^2 - 2\inprod{u_k}{x_k + S^{-1}u_k - z} \to \nu
		\end{align*}
		which concludes the proof of \cref{thm:itm:weakconv}.
	\end{proof}
\end{thm}

\section{Explicit Iterate Momentum}\label{sec:add-momentum}
Consider the following variant of \cref{alg:nofob_mom} that adds an additional scaled momentum term $\gamma_k^{-1}\theta S(x_{k}-x_{k-1})$.
\begin{algorithm}[H]
	\caption{Nonlinear Forward-Backward with Momentum Correction and Additional Iterate Momentum}
	\label{alg:nofob_mom_add_mom}
	Consider problem \cref{eq:prob} and let $S$ be such that \cref{ass:prob-assump} is satisfied.
	With $x_0, x_{-1}, u_0 \in \Hprim$, for all $k\in\N$ iteratively perform
	\begin{align*}
		x_{k+1}
		&=(M_k + A)^{-1}(M_kx_k - Cx_k + \gamma_k^{-1}u_k + \gamma_k^{-1}\theta S(x_{k} - x_{k-1})), \\
		u_{k+1} &= (\gamma_k M_{k} - S)x_{k+1} - (\gamma_k M_{k} - S)x_{k},
	\end{align*}
	where $M_k\colon\Hprim\to\Hprim$, $\gamma_k > 0$ and $\theta < 1$.
\end{algorithm}
We will show in \cref{cor:add-mom-conv} that there always exists a $\theta \neq 0$---possibly negative---such that if \cref{alg:nofob_mom} converges, so does \cref{alg:nofob_mom_add_mom}.
This shows that it is always possible to add this type of iterate momentum to an instance of \cref{alg:nofob_mom}.
We will use this in the next section to develop a new momentum variant of the Forward-Half-Reflected-Backward method.
Although it might seem like \cref{alg:nofob_mom_add_mom} has more degrees of freedom than \cref{alg:nofob_mom}, this is not the case.
In fact, \cref{alg:nofob_mom_add_mom} is equivalent to \cref{alg:nofob_mom}---we show and use this in the proofs below.
\Cref{alg:nofob_mom_add_mom} is therefore first and foremost a tool for adding momentum to an already known instance of \cref{alg:nofob_mom} and the usefulness comes via the following corollary that gives an explicit convergence condition.
\begin{cor}\label{cor:add-mom-conv}
	Let \cref{ass:prob-assump,ass:alg-assump} hold and let $\theta < 1$.
	If there exists an $\varepsilon > 0$ such that
	\begin{equation}\label{eq:cor:add-mom-cond}
		1 - \theta - 2|\theta| - L_{k-1} - L_{k} - \gamma_k \tfrac{\ell}{2} \geq \varepsilon
	\end{equation}
	for all $k\in\N_+$, then \cref{alg:nofob_mom_add_mom} satisfies the following as $k\to\infty$:
	\begin{enumerate}[(i)]
		\itemsep0em
		\item $x_{k+1} - x_k \to 0$,
		\item $u_k \to 0$,
		\item $(A+C)x_{k+1} \ni M_kx_k - M_kx_{k+1} + \gamma_k^{-1}u_{k} + \gamma_k^{-1}\theta S(x_k-x_{k-1}) + Cx_{k+1} -  Cx_{k} \to 0$,
		\item $x_{k} \rightharpoonup x^\star$ for some $x^\star \in \zer (A+C)$.
	\end{enumerate}
	\begin{proof}
		By defining $\hat{\gamma}_k = \frac{\gamma_k}{1-\theta}$ and $\hat{u}_{k+1} = \frac{1}{1-\theta}u_{k+1} + \frac{\theta}{1-\theta}S(x_{k+1} - x_{k})$, the update of \cref{alg:nofob_mom_add_mom} can equivalently be written as
		\begin{equation}\label{eq:add_mom_equivalent}
			\begin{aligned}
				x_{k+1} &=(M_k + A)^{-1}(M_kx_k - Cx_k + \hat{\gamma}_k^{-1}\hat{u}_k), \\
				\hat{u}_{k+1} &= (\hat{\gamma}_k M_{k} - S)x_{k+1} - (\hat{\gamma}_k M_{k} - S)x_{k}
			\end{aligned}
		\end{equation}
		which is the same as the update of \cref{alg:nofob_mom} but with $\hat{\gamma}_k$ and $\hat{u}_k$ instead of $\gamma_k$ and $u_k$ respectively.
		\Cref{alg:nofob_mom_add_mom} is therefore equivalent to \cref{alg:nofob_mom}.
		Since, by \cref{ass:alg-assump}, $\gamma_kM_k-S$ is $L_k$-Lipschitz w.r.t.\ $S$ and
		\begin{align*}
			\hat{\gamma}_kM_k - S = \tfrac{1}{1-\theta}(\gamma_kM_k - S) + \tfrac{\theta}{1-\theta}S
		\end{align*}
		we conclude that $\hat{\gamma}_kM_k - S$ is $\frac{L_k + |\theta|}{1-\theta}$-Lipschitz continuous w.r.t.\ $S$.
		We further have that $\hat{\gamma}_k = \frac{\gamma_k}{1-\theta} \geq \frac{\gamma}{1-\theta} >  0$ and \cref{ass:alg-assump} is therefore satisfied for \cref{eq:add_mom_equivalent}.
		The convergence condition \cref{eq:thm:conv-cond} from \cref{thm:main-conv} for the algorithm update \cref{eq:add_mom_equivalent} is then that there exists an $\epsilon > 0$ such that
		\begin{gather*}
			1 - \tfrac{L_{k-1} + |\theta|}{1-\theta} - \tfrac{L_{k} + |\theta|}{1-\theta} - \tfrac{\gamma_k}{1-\theta} \tfrac{\ell}{2} \geq \epsilon.
		\end{gather*}
		Multiplication of both sides by $1 - \theta$ and noting that $\theta < 1$ gives the equivalent condition that there exists an $\varepsilon > 0$ such that
		\begin{gather*}
			1 - \theta - 2|\theta| - L_{k-1} - L_{k} - \gamma_k \tfrac{\ell}{2} \geq \varepsilon.
		\end{gather*}
		The convergence results for \cref{alg:nofob_mom_add_mom} follow directly from \cref{thm:main-conv}.
	\end{proof}
\end{cor}

\begin{cor}\label{cor:add-mom-possible}
	If the conditions of \cref{thm:main-conv} hold---implying that \cref{alg:nofob_mom} converges to a solution of \cref{eq:prob}---there exists a $\theta\neq 0$ with $\theta < 1$ such that the conditions of \cref{cor:add-mom-conv} also hold and the additional momentum method in \cref{alg:nofob_mom_add_mom} converges to a solution of \cref{eq:prob}.
	\begin{proof}
		The assumptions on $A$, $C$, $S$, $M_k$, and $\gamma_k$ of \cref{thm:main-conv} and \cref{cor:add-mom-conv} are identical so it is enough to conclude that there exists a $\theta \neq 0$ and $\theta<1$ such that convergence condition \cref{eq:cor:add-mom-cond} of \cref{cor:add-mom-conv} is implied by the conditions of \cref{thm:main-conv}.
		Since \cref{thm:main-conv} holds, we know that
		\begin{gather*}
			1 - L_{k-1} - L_{k} - \gamma_k \tfrac{\ell}{2} \geq \epsilon > 0.
		\end{gather*}
		Since $\epsilon > 0$ there exist a $\theta$ such that $-\tfrac{1}{2}\epsilon < \theta < \tfrac{1}{6}\epsilon$, $\theta \neq 0$, and $\theta < 1$.
		Selecting such a $\theta$ yields $\tfrac{1}{2}\epsilon > \theta + 2|\theta| > 0$ and
		\begin{gather*}
			1 - L_{k-1} - L_{k} - \gamma_k \tfrac{\ell}{2} \geq \epsilon > \tfrac{1}{2}\epsilon + \theta + 2|\theta| > 0.
		\end{gather*}
		Subtracting $\theta + 2|\theta|$ and defining $\varepsilon = \tfrac{1}{2}\epsilon$ yield
		\begin{gather*}
			1 - \theta - 2|\theta| - L_{k-1} - L_{k} - \gamma_k \tfrac{\ell}{2} \geq \varepsilon > 0
		\end{gather*}
		which is the convergence condition \cref{eq:cor:add-mom-cond} for \cref{alg:nofob_mom_add_mom}.
	\end{proof}
\end{cor}

\begin{rem}\label{rem:add-mom-tract}
	From \cref{cor:add-mom-possible}, we know that we can always add momentum to an instance of \cref{alg:nofob_mom} and still get a convergent algorithm.
	In most cases, the per iteration computational cost of the momentum variant is similar to that of the basic method.
	However, it is possible for the momentum variant not to be tractable.
	More precisely, it might not be possible to cheaply evaluate $(M_k + A)^{-1}$ at $M_kx_k - Cx_k + \gamma_k^{-1}u_k + \gamma_k^{-1}\theta S(x_{k} - x_{k-1})$ even though it can be cheaply evaluated at $M_kx_k - Cx_k + \gamma_k^{-1}u_k$.
	We will show an example of this in \cref{alg:primdual-res}.
	For \cref{alg:primdual-res}, this problem can be handled by introducing a $\theta$-dependent term in the nonlinear kernel.
\end{rem}

\section{Forward-Half-Reflected-Backward with Iterate Momentum}\label{sec:fhrb}
We present the forward-half-reflected-backward with iterate momentum algorithm in \Cref{alg:fhrb-mom} as a special case of \Cref{alg:nofob_mom_add_mom}. 
\begin{algorithm}
	\caption{Forward-Half-Reflected-Backward with Iterate Momentum}
	\label{alg:fhrb-mom}
	Consider problem \cref{eq:fhrb-prob}.
	With $x_0, x_{-1} \in \Hprim$ and $\alpha_{-1} > 0$, for all $k\in\N$ iteratively perform
	\begin{align*}
		\bar{x}_k &= x_k + \theta(x_{k}-x_{k-1}), \\
		x_{k+1} &=(\Id+\alpha_kB)^{-1}(\bar{x}_k - \alpha_kCx_k - (\alpha_k+\alpha_{k-1})Dx_k + \alpha_{k-1}Dx_{k-1} )
	\end{align*}
	where $\alpha_k > 0$ and $\theta < 1$.
\end{algorithm}
It is a method for finding $x\in\Hprim$ such that
\begin{equation}\label{eq:fhrb-prob}
0 \in Bx + Dx + Cx
\end{equation}
for which the following assumption holds.
\begin{ass}\label{ass:fhrb-prob-ass}
	The operators of \cref{eq:fhrb-prob} satisfy:
	\begin{enumerate}[(i)]
		\itemsep0em
		\item $B\colon\Hprim \to 2^\Hprim$ is maximally monotone.
		\item $D\colon\Hprim\to\Hprim$ is $\delta$-Lipschitz continuous.
		\item $B + D$ is maximally monotone.
		\item $C\colon\Hprim\to\Hprim$ is $\beta^{-1}$-cocoercive.
		\item $\zer (B+D+C) \neq \emptyset$.
	\end{enumerate}
If $C = 0$, we set $\beta = 0$.
\end{ass}

By letting $A = B+D$, problem \cref{eq:fhrb-prob} can be seen as an instance of our standard problem formulation \cref{eq:prob}.
By letting $S=\Id$, \Cref{ass:fhrb-prob-ass} implies that \Cref{ass:prob-assump} holds with $\ell=\beta$.
With these choices, \Cref{alg:fhrb-mom} is obtained from \cref{alg:nofob_mom_add_mom} by choosing $M_k = \alpha_k^{-1}\Id - D$ and $\gamma_k = \alpha_k$ for some step-size $\alpha_k > 0$.
The backward step of the algorithm becomes
\begin{align*}
	(M_k + A)^{-1}
	= (\alpha_k^{-1}\Id - D + B + D)^{-1}
	= (\Id + \alpha_k B)^{-1}\circ\alpha_k\Id.
\end{align*}
Note, the backward step is independent of $D$ and the algorithm will, as we will show next, only depend on $D$ through the forward step.
The operator $\gamma_kM_k-S$ used in the correction term becomes
\begin{align*}
	\gamma_kM_k-S = \alpha_k(\alpha_k^{-1}\Id - D) - \Id = -\alpha_kD,
\end{align*}
and the complete forward step with momentum correction is
\begin{align*}
	&M_kx_k - Cx_k + \gamma_k^{-1}u_k +\gamma_k^{-1}\theta S(x_k-x_{k-1})\\
	& = \alpha_k^{-1}x_k - Dx_k - Cx_k - \alpha_k^{-1}(\alpha_{k-1}Dx_k - \alpha_{k-1}Dx_{k-1})+\alpha_k^{-1}\theta (x_k-x_{k-1})\\
 & = \alpha_k^{-1}\bar{x}_k - Dx_k - Cx_k - \alpha_k^{-1}(\alpha_{k-1}Dx_k - \alpha_{k-1}Dx_{k-1}),
\end{align*}
where $\bar{x}_k=x_k+\theta(x_k-x_{k-1})$.
Combining the backward and forward steps yields the full \cref{alg:fhrb-mom}.
In this algorithm, we do not need to evaluate both $M_{k-1}x_{k}$ and $M_kx_k$ from scratch since we can reuse the potentially expensive computation of $Dx_k$.
\begin{cor}\label{cor:fhrb-mom-conv}
	Let \cref{ass:fhrb-prob-ass} hold and consider problem \cref{eq:fhrb-prob} and \cref{alg:fhrb-mom}.
	If there exists $\epsilon > 0$ such that
	\begin{align*}
		\epsilon \leq \alpha_k ,\quad \alpha_{k}\delta + \alpha_{k+1}(\delta + \tfrac{\beta}{2}) \leq (1-\theta-2|\theta|) - \epsilon
	\end{align*}
	for all $k\in\N$, then $x_k \rightharpoonup x^\star$ where $x^\star$ is a solution to \cref{eq:fhrb-prob}.
	\begin{proof}
		After \cref{ass:fhrb-prob-ass}, we concluded that \cref{ass:prob-assump} holds for the reformulation of \cref{eq:fhrb-prob} into \cref{eq:prob} via $A=B+D$.
		\Cref{ass:alg-assump} also holds since $\gamma_k = \alpha_k \geq \epsilon > 0$ and $\gamma_kM_k - S = -\alpha_k D$ is $\alpha_k\delta$-Lipschitz continuous.
		Inserting $\gamma_k$, $\beta$, $\delta$, and $\theta$ into \cref{eq:cor:add-mom-cond} of \cref{cor:add-mom-conv} then directly gives the step-size condition and the results follow from the corollary.
	\end{proof}
\end{cor}

The forward-half-reflected-backward (FHRB) method and its special case, the forward-reflected-backward (FRB) method\footnote{FHRB was referred to as a three-operator splitting variant of FRB in the original work.} in \cite{malitskyForwardBackwardSplittingMethod2020}, are special cases of \Cref{alg:fhrb-mom}. They are obtained by setting $\theta=0$ (FHRB) and $\theta=0$ and $C=0$ (FRB).
Our analysis assumes that $B$ and $B+D$ are maximally monotone. In \cite{malitskyForwardBackwardSplittingMethod2020}, they instead assume that $B$ and $D$ are both maximally monotone which implies that $B+D$ is maximally monotone since $D$ is also Lipschitz continuous with full domain.
Our assumptions are slightly more general since we can allow for non-monotone $D$ as long as $B$ can compensate for it.

Our step-size conditions are slightly relaxed compared to the ones in \cite{malitskyForwardBackwardSplittingMethod2020}.
Our conditions match these when a constant step-size $\alpha_k = \alpha$ is chosen.
However, the original work only provides convergence conditions for non-constant step-sizes in the FRB case, i.e., $C=0$.
In that case, \cite{malitskyForwardBackwardSplittingMethod2020} proved convergence if $\epsilon \leq 2\alpha_k \leq \delta^{-1} - \epsilon$ for some $\epsilon > 0$ and all $k\in\N$ which is slightly more restrictive than our condition.

When $C = 0$, \Cref{alg:fhrb-mom} is \cite[Equation 4.1]{malitskyForwardBackwardSplittingMethod2020} without relaxation and when $D =  0$ it is forward-backward splitting with momentum.
Both of these special cases have been shown to converge under certain conditions but our results expand these conditions in both settings.
In the FRB with momentum case, \cref{cor:fhrb-mom-conv} allows for step-sizes that depend on the iteration index $k$ while \cite[Theorem 4.3]{malitskyForwardBackwardSplittingMethod2020} only allows for constant step-size, $\alpha_k = \alpha$ for all $k \in \N$.
In the forward-backward with momentum case, \cref{cor:fhrb-mom-conv} makes it possible to find a convergent step-size $\alpha_k$ for all $\theta \in (-1,\frac{1}{3})$, which is the only result we know of that allows for negative momentum.
This is especially interesting considering that the magnitude of negative momentum is allowed to be larger than the magnitude of positive momentum.
Our upper bound on the momentum matches other results in the literature for weak sequence convergence---\cite{malitskyForwardBackwardSplittingMethod2020} when $C = 0$, \cite{alvarezInertialProximalMethod2001} when $C=D=0$, and \cite{moudafiConvergenceSplittingInertial2003} when $C \neq 0$ and $D=0$.\footnote{
	The work in \cite{moudafiConvergenceSplittingInertial2003} does not present an explicit convergence condition for a fixed choice of $\theta$.
	Instead, they present a criterion for selecting an iteration dependent $\theta_k$ adaptively.
	However, in a remark they mention results from \cite{alvarezInertialProximalMethod2001} which, when combined with their results, yield a convergence criteria for a fixed choice of $\theta$.
}
In the gradient-descent case, larger upper bounds on $\theta$ and $\alpha_k$ have been shown to work \cite{ghadimiGlobalConvergenceHeavyball2015}.
These results guarantee ergodic convergence of function values and are not applicable to general monotone inclusion problems.

\begin{rem}
	The same nonlinear kernel that in this case generates FHRB and FRB yields the forward-backward-half-forward \cite{briceno-ariasForwardBackwardHalfForwardAlgorithm2018} and forward-backward-forward \cite{tsengModifiedForwardBackwardSplitting2000} methods when used in the nonlinear forward-backward scheme with projection correction \cite{giselssonNonlinearForwardBackwardSplitting2021}.
	The two sets of algorithms can therefore be seen to have the same nonlinear forward-backward step but with different correction methods to guarantee convergence.
	Due to the momentum correction's reuse of old information, FHRB and FRB have cheaper per-iteration costs compared to the projection correction counterparts.
\end{rem}

\section{Two Novel Primal-Dual Methods}\label{sec:primdual}

We will present two new primal-dual methods for solving the problem of finding $y \in \Hsec$ such that
\begin{equation}\label{eq:primdualprob-prim}
	0\in By + (V^*\circ D\circ V)y + Ey + Fy
\end{equation}
where the following assumptions hold.
\begin{ass}\label{ass:primdualprob-ass}
Let $\Hsec$ and $\Hdual$ be real Hilbert spaces. The operators of \cref{eq:primdualprob-prim} satisfy:
	\begin{enumerate}[(i)]
		\itemsep0em
		\item $B\colon\Hsec \to 2^\Hsec$ and $D\colon \Hdual\to 2^\Hdual$ are maximally monotone.
		\item $E\colon\Hsec\to\Hsec$ is monotone and $\delta$-Lipschitz continuous.
		\item $F\colon\Hsec\to\Hsec$ is $\beta^{-1}$-cocoercive.
		\item $V\colon\Hsec\to\Hdual$ is linear and bounded.
		\item $\zer (B + (V^*\circ D\circ V) + E + F) \neq \emptyset$.
	\end{enumerate}
If $F = 0$, we set $\beta = \beta^{-1} = 0$.
\end{ass}
By a primal-dual method, we mean a method that, instead of solving \cref{eq:primdualprob-prim} directly, solves the equivalent primal-dual problem of finding $y\in \Hsec$ and $z \in \Hdual$ such that%
\begin{equation}\label{eq:primdualprob}
	0 \in
	\begin{cases}
		By + V^*z + Ey + Fy \\
		D^{-1}z - Vy.
	\end{cases}
\end{equation}
The two primal-dual methods are derived by reformulating this primal-dual problem into our standard form \cref{eq:prob} and then applying \cref{alg:nofob_mom} with different sets of design parameters.
There is no unique way of reformulating \cref{eq:primdualprob} into \cref{eq:prob} but we set $\Hprim = \Hsec\times\Hdual$ and define, with some abuse of block matrix notation, $A\colon\Hsec\times\Hdual \to 2^{\Hsec\times\Hdual}$ and $C\colon\Hsec\times\Hdual \to \Hsec\times\Hdual$ as
\begin{equation}\label{eq:primdual-operators}
	A
	= \underbrace{\begin{bmatrix} B & 0 \\ 0 & D^{-1} \end{bmatrix}}_{\widehat{A}}
	+ \underbrace{\begin{bmatrix} E & 0 \\ 0 & 0 \end{bmatrix}}_{\widehat{E}}
	+ \underbrace{\begin{bmatrix} 0 & V^* \\ -V & 0 \end{bmatrix}}_{\widehat{V}}
	\quad \text{and} \quad
	C = \begin{bmatrix} F & 0 \\ 0 & 0 \end{bmatrix}.
\end{equation}
Assuming $A+C$ has at least one zero, these operators satisfy \cref{ass:prob-assump} since $A = \widehat{A} + \widehat{E} + \widehat{V}$ is the sum of a maximally monotone operator $\widehat{A}$ and two maximally monotone operators $\widehat{E}$ and $\widehat{V}$ with full domains.
The properties of $\widehat{A}$, $\widehat{E}$, and $\widehat{V}$ are results of the following: maximal monotonicity of $B$ and $D$; monotonicity and Lipschitz continuity of $E$; and the skew-adjointness and linearity of $\widehat{V}$.
The first assumption of \cref{ass:prob-assump} is then satisfied and the second assumption regarding the cocoercivity of $C$ is easily verified in the standard metric of $\Hsec\times\Hdual$.
However, the algorithms in \cref{sec:primdual-tri,sec:primdual-res} will use different scaling operators $S$ and we will therefore defer the derivation of more precise cocoercivity constants to the respective sections since the constants depend on $S$.

\subsection{Primal-Dual Method with Block-Triangular Resolvent}\label{sec:primdual-tri}
To derive our first primal-dual algorithm, we decompose the iterates of \cref{alg:nofob_mom} as $x_k=(y_k,z_k)$ with $y_k\in\Hsec$ and $z_k\in\Hdual$ for all $k\in\N$.
The algorithm is given by the following design parameters
\begin{equation}\label{eq:primdual-tri-params}
	S =
	\begin{bmatrix} \Id & -\tau V^* \\ -\tau V & \tau\sigma^{-1}\Id \end{bmatrix}
	, \quad
	M_k =
	\underbrace{\begin{bmatrix} \tau^{-1}\Id & 0 \\ -\lambda_k V & \sigma^{-1}\Id \end{bmatrix}}_{\widehat{M}_k}
	- \widehat{E}
	- \widehat{V}
	\quad \text{and} \quad
	\gamma_k = \tau
\end{equation}
where $\tau,\sigma > 0$ such that $\tau\sigma\norm{V}^2 < 1$ and $\lambda_k \in \R$ for all $k\in\N$.
The assumption on $\tau$ and $\sigma$ guarantees that $S\in\mathcal{P}(\Hsec\times\Hdual)$.
The forward step operator and the correction operator are
\begin{align*}
	M_k - C =
	\begin{bmatrix}
		\tau^{-1}\Id  - E - F & - V^*\\
		(1-\lambda_k)V & \sigma^{-1}\Id
	\end{bmatrix}
	,\quad
	\gamma_kM_k - S
	=
	\tau
	\begin{bmatrix}
		- E & 0 \\
		(2 - \lambda_k) V & 0
	\end{bmatrix}
	.
\end{align*}
Inserting these operators into the complete forward step with correction,
\begin{equation*}
	\begin{aligned}
		(\hat{y}_k,\hat{z}_k) \coloneqq\;
		& M_k(y_k,z_k) - C(y_k,z_k) + \gamma_k^{-1}(\gamma_{k-1}M_{k-1}-S)(y_k,z_k) \\
		& - \gamma_k^{-1}(\gamma_{k-1}M_{k-1}-S)(y_{k-1},z_{k-1}),
	\end{aligned}
\end{equation*}
where $(\hat{y}_k,\hat{z}_k)\in\Hsec\times\Hdual$, yields
\begin{align*}
	\hat{y}_k &= \tau^{-1}y_k -V^*z_k - (2Ey_k -Ey_{k-1}) - Fy_k, \\
	\hat{z}_k &= \sigma^{-1}z_k + (1-\lambda_k)Vy_k + (2-\lambda_{k-1})V(y_k-y_{k-1}).
\end{align*}
What remains to compute is the backward step.
The kernel $M_k$ is designed to cancel out the $\widehat{E}$ and $\widehat{V}$ terms, making only the forward step depend on these operators,
\begin{align*}
	(M_k + A)^{-1} = (\widehat{M}_k - \widehat{E} - \widehat{V} + \widehat{A} + \widehat{E} + \widehat{V})^{-1} = (\widehat{M}_k + \widehat{A})^{-1}.
\end{align*}
This is the inverse of a lower block triangular operator and it can therefore be computed with back substitution according to
\begin{align*}
	&(y_{k+1},z_{k+1}) = (\widehat{M}_k + \widehat{A})^{-1}(\hat{y}_k,\hat{z}_k) \\
	&\quad\qquad\iff\quad
	(\hat{y}_k,\hat{z}_k) \in (\widehat{M}_k + \widehat{A})(y_{k+1},z_{k+1}) \\
	&\quad\qquad\iff\quad
	\left\{
	\begin{alignedat}{3}
		\hat{y}_k &\in (\tau^{-1}\Id + B)y_{k+1} \\
		\hat{z}_k &\in -\lambda_k Vy_{k+1} + (\sigma^{-1}\Id + D^{-1})z_{k+1}
	\end{alignedat}
	\right.
	\\
	&\quad\qquad\iff\quad
	\left\{
	\begin{alignedat}{3}
		y_{k+1} & {}={} (\Id + \tau B)^{-1} (\tau \hat{y}_k) \\
		z_{k+1} & {}={} (\Id + \sigma D^{-1})^{-1} (\sigma \hat{z}_k + \sigma \lambda_k V y_{k+1}).
	\end{alignedat}
	\right.
\end{align*}
Inserting the expressions for $\hat{y}_k$ and $\hat{z}_k$ results in the following algorithm.
\begin{algorithm}[H]
	\caption{Primal-Dual Method with Block Triangular Resolvent}
	\label{alg:primdual-tri}
	Consider problem \cref{eq:primdualprob-prim}.
	With $y_0, y_{-1} \in \Hsec$, $z_0\in\Hdual$ and $\lambda_{-1} \in \R$, for all $k\in\N$ iteratively perform
	\begin{align*}
		y_{k+1} &= (\Id + \tau B)^{-1} (y_k - \tau V^*z_k - \tau ( 2 Ey_k - Ey_{k-1}) - \tau Fy_k), \\
		v_{k+1} &= \lambda_{k}(y_{k+1} - y_k) + (2-\lambda_{k-1})(y_k - y_{k-1}), \\
		z_{k+1} &= (\Id + \sigma D^{-1})^{-1}(z_k + \sigma V(y_k + v_{k+1})),
	\end{align*}
	where $\tau,\sigma > 0$ and $\lambda_k \in \R$.
\end{algorithm}

Due to the lower block-triangular structure of the operator in the backward step, the primal update of $y_{k+1}$ is independent of the dual update of $z_{k+1}$ but the opposite statement does not hold in general.
This dependency is controlled by $\lambda_k$ and manifests itself as a correction $v_{k+1}$ added to the primal iterate used in the dual update.
When $\lambda_k = \lambda_{k-1}$, the correction $v_{k+1}$ is an affine combination of an extrapolation step based either on the current or previous primal update, see \cref{fig:primdual-tri-corr}.
When $\lambda_k \neq \lambda_{k-1}$, the correction can be an arbitrary linear combination of the two different extrapolations.
However, the choice of the sequence $(\lambda_k)_{k\in\N}$ will affect the range of allowed step-sizes.
The more $\lambda_k$ differs from $2$, the smaller the upper bound on the step-sizes is in the following convergence result.

\begin{figure}[t]
	\centering
	\begin{tikzpicture}
		\pgfmathsetmacro{\varlambda}{1.2}

		\coordinate (y0) at (0,0) {};
		\coordinate (y1) at (2,0) {};
		\coordinate (y2) at (3,-1) {};
		\coordinate (l0) at ($(y1)+2*(y1)-2*(y0)$) {};
		\coordinate (l2) at ($(y1)+2*(y2)-2*(y1)$) {};
		\coordinate (yhat) at ($(y1)+\varlambda*(y2)-\varlambda*(y1)+2*(y1)-2*(y0)-\varlambda*(y1)+\varlambda*(y0)$) {};

		\node at (y0) [above = 0] {$y_{k-1}$};
		\node at (y1) [above = 0] {$y_{k}$};
		\node at (y2) [below left = 0] {$y_{k+1}$};
		\node at (l0) [below right = 0] {$y_k + v_{k+1}$, ($\lambda_k = \lambda_{k-1} = 0$)};
		\node at (l2) [below right = 0] {$y_k + v_{k+1}$, ($\lambda_k = \lambda_{k-1} = 2$)};
		\node at (yhat) [below right = 0] {$y_k + v_{k+1}$, ($\lambda_k = \lambda_{k-1} = \varlambda$)};

		\draw[line width=1pt, -stealth] (y0)--(y1) {};
		\draw[line width=1pt, -stealth] (y1)--(y2) {};
		\draw[line width=1pt, dotted] ($(l0)!-.15!(l2)$)--($(l0)!1.20!(l2)$) {};
		\draw[line width=1pt, -stealth, dashed] (y1)--(l0) {};
		\draw[line width=1pt, -stealth, dashed] (y1)--(l2) {};
		\draw[line width=1pt, -stealth, dashed] (y1)--(yhat) {};
	\end{tikzpicture}
	\caption{Update of the corrected primal iterate $y_k + v_{k+1}$ in \cref{alg:primdual-tri}.}
	\label{fig:primdual-tri-corr}
\end{figure}
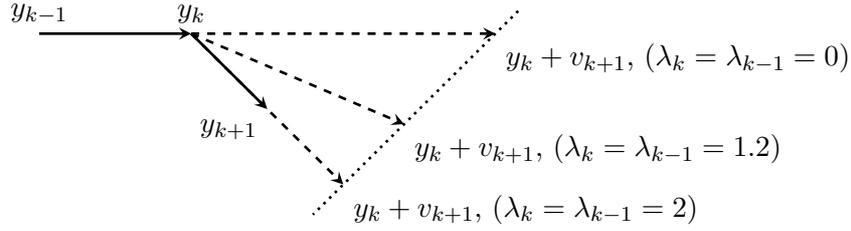

\begin{cor}\label{cor:primdual-tri-conv}
	Let \cref{ass:primdualprob-ass} hold and consider problem \cref{eq:primdualprob-prim} and \cref{alg:primdual-tri}.
	If there exists $\epsilon > 0$ such that
	\begin{align*}
		\tau\sigma\norm{V}^2 + (|2-\lambda_k| + |2-\lambda_{k+1}|)\sqrt{\tau\sigma}\norm{V} + \tau(2\delta + \tfrac{1}{2}\beta) < 1-\epsilon
	\end{align*}
	for all $k\in\N$, then $y_k \rightharpoonup y^\star$ and $z_k \rightharpoonup z^\star$ where $y^\star$ is a solution to \cref{eq:primdualprob-prim} and $(y^\star,z^\star)$ is a solution to \cref{eq:primdualprob}.
\end{cor}

Before proceeding to the proof of \cref{cor:primdual-tri-conv}, we present the following lemma on which the proof relies.
\begin{lem}\label{lem:primdual-tri-scaling}
	Let $S\in\mathcal{P}(\Hsec\times\Hdual)$ be from \cref{eq:primdual-tri-params}.
	The inverse of $S$ satisfies
	\begin{align*}
		S^{-1} =
		\begin{bmatrix} (\Id - \tau\sigma V^*V)^{-1} & 0 \\ 0 & (\Id - \tau\sigma VV^*)^{-1} \end{bmatrix}
		\begin{bmatrix} \Id & \sigma V^* \\ \sigma V & \tau^{-1}\sigma\Id \end{bmatrix}.
	\end{align*}
	The following inequalities hold for all $y \in \Hsec$ and $z\in \Hdual$:
	\begin{gather*}
		\norm{(y,0)}_{S^{-1}}^2 \leq \tfrac{1}{1-\tau\sigma\norm{V}^2}\norm{y}^2
		,\quad
		\norm{(0,z)}_{S^{-1}}^2 \leq \tfrac{\tau^{-1}\sigma}{1-\tau\sigma \norm{V}^2}\norm{z}^2\\
		\text{and} \quad
		\norm{y}^2 \leq \tfrac{1}{1-\tau\sigma\norm{V}^2} \norm{(y,z)}_{S}^2.
	\end{gather*}

	\begin{proof}
		The inverse is easily verified and we note that, since $\tau\sigma\norm{V}^2 < 1$ by assumption, $\Id - \tau \sigma V^*V \in \mathcal{P}(\Hsec)$ and $\Id - \tau \sigma VV^* \in \mathcal{P}(\Hdual)$  and hence they are invertible.
		Let $y\in\Hsec$, then
		\begin{align*}
			\norm{(y,0)}_{S^{-1}}^2
			&= \inprod{(\Id-\tau\sigma V^*V)^{-1}y}{y} \\
			&\leq \norm{(\Id-\tau\sigma V^*V)^{-1}}\norm{y}^2 \\
			&\leq \tfrac{1}{1-\tau\sigma \norm{V}^2}\norm{y}^2
		\end{align*}
		which proves the first inequality of the lemma.
		The last step holds since $1 > \tau\sigma \norm{V}^2$.
		Let $z\in\Hdual$, then
		\begin{align*}
			\norm{(0,z)}_{S^{-1}}^2
			&= \tau^{-1}\sigma \inprod{(\Id-\tau\sigma VV^*)^{-1}z}{z} \\
			&\leq \tau^{-1}\sigma \norm{(\Id-\tau\sigma VV^*)^{-1}}\norm{z}^2 \\
			&\leq \tfrac{\tau^{-1}\sigma}{1-\tau\sigma \norm{V}^2}\norm{z}^2
		\end{align*}
		which proves the second inequality of the lemma.
		Again, the last step holds since $1 > \tau\sigma \norm{V}^2$.
		Let $y\in\Hsec$ and $z\in\Hdual$, then
		\begin{align*}
			\norm{(y,z)}_S^2
			&= \norm{y}^2 + \tau\sigma^{-1}\norm{z}^2 - 2\tau\inprod{Vy}{z} \\
			&\geq \norm{y}^2 + \tau\sigma^{-1}\norm{z}^2 - \tau(\sigma\norm{V}^2\norm{y}^2 + \sigma^{-1}\norm{z}^2) \\
			&= (1 - \tau\sigma\norm{V}^2)\norm{y}^2
		\end{align*}
		which proves the third inequality of the lemma.
	\end{proof}
\end{lem}

\begin{proof}[\textbf{Proof of \cref{cor:primdual-tri-conv}}]
	As previously stated, the choice of $A$ and $C$ in \cref{eq:primdual-operators} satisfies \cref{ass:prob-assump} since we assume that a solution exists.
	What remains to verify of \cref{ass:prob-assump} is to derive a cocoercivity constant of $C$.
	The first inequality of \cref{lem:primdual-tri-scaling} directly gives
	\begin{align*}
		\norm{C(y,z) - C(y',z')}_{S^{-1}}^2
		&\leq \tfrac{1}{1-\tau\sigma\norm{V}^2}\norm{Fy-Fy'}^2 \\
		&\leq \tfrac{\beta}{1-\tau\sigma\norm{V}^2}\inprod{Fy-Fy'}{y-y'} \\
		&= \tfrac{\beta}{1-\tau\sigma\norm{V}^2}\inprod{C(y,z)-C(y',z')}{(y,z)-(y',z')}
	\end{align*}
	for all $(y,z),(y',z')\in\Hsec\times\Hdual$.
	Hence, $C$ is $\ell^{-1}$-cocoercive w.r.t.\ $S$ with $\ell = \frac{\beta}{1-\tau\sigma\norm{V}^2}$.
	Note that we can set $\ell = 0$ if $F=0$.

	The assumptions placed on the design parameters, \cref{ass:alg-assump}, also need to hold.
	For item \cref{ass:item:alg-assump-gamma} of \cref{ass:alg-assump}, we directly see that $\gamma_k = \tau > 0$.
	We prove \cref{ass:item:alg-assump-M} of \cref{ass:alg-assump}, the Lipschitz continuity of
	\begin{align*}
		\gamma_kM_k - S = \tau(\widehat{M}_k- \widehat{V}) - S - \tau\widehat{E},
	\end{align*}
	by showing Lipschitz continuity of $\tau\widehat{E}$ and of $\tau(\widehat{M}_k-\widehat{V})-S$ separately.
	The Lipschitz continuity of $\gamma_kM_k - S$ then follows from the Lipschitz continuity of a sum of Lipschitz continuous operators.
	Starting with $\tau\widehat{E}$ and using the first and third inequalities from \cref{lem:primdual-tri-scaling} and the Lipschitz continuity of $E$ gives
	\begin{align*}
		\norm{\widehat{E}(y,z)-\widehat{E}(y',z')}_{S^{-1}}^2
		&\leq \tfrac{1}{1 - \tau\sigma\norm{V}^2} \norm{Ey-Ey'}^2 \\
		&\leq \tfrac{\delta^2}{1 - \tau\sigma\norm{V}^2} \norm{y-y'}^2 \\
		&\leq \tfrac{\delta^2}{(1 - \tau\sigma\norm{V}^2)^2} \norm{(y,z)-(y',z')}_S^2
	\end{align*}
	for all $(y,z),(y',z')\in\Hsec\times\Hdual$.
	The term $\tau\widehat{E}$ is therefore $\frac{\tau\delta}{1-\tau\sigma\norm{V}^2}$-Lipschitz continuous w.r.t.\ $S$.
	For $\tau(\widehat{M}_k-\widehat{V})-S$, we first note that
	\begin{align*}
		\tau(\widehat{M}_k-\widehat{V}) - S = \begin{bmatrix} 0 & 0 \\ \tau(2-\lambda_k) V & 0 \end{bmatrix}
	\end{align*}
	and we can use the second inequality of \cref{lem:primdual-tri-scaling}:
	\begin{align*}
		\norm{(\tau(\widehat{M}_k-\widehat{V})-S)(y,z)}_{S^{-1}}^2
		&\leq \tfrac{\tau^{-1}\sigma}{1-\tau\sigma \norm{V}^2}\norm{\tau(2-\lambda_k)Vy}^2\\
		&\leq (2-\lambda_k)^2\tfrac{\tau\sigma\norm{V}^2}{1 - \tau\sigma\norm{V}^2}\norm{y}^2 \\
		&\leq (2-\lambda_k)^2\tau\sigma\norm{V}^2\tfrac{1}{(1 - \tau\sigma\norm{V}^2)^2} \norm{(y,z)}_S^2
	\end{align*}
	for all $(y,z)\in\Hsec\times\Hdual$.
	The operator $\tau(\widehat{M}_k-\widehat{V})-S$ is therefore Lipschitz continuous w.r.t.\ $S$ with constant $|2-\lambda_k|\sqrt{\tau\sigma}\norm{V}\tfrac{1}{1-\tau\sigma\norm{V}^2}$.
	Adding these two Lipschitz constants yields that $\gamma_kM_k-S$ is $L_k$-Lipschitz continuous w.r.t.\ $S$ where
	\begin{align*}
		L_k = \tfrac{1}{1-\tau\sigma\norm{V}^2}(|2-\lambda_k|\sqrt{\tau\sigma}\norm{V} + \tau\delta),
	\end{align*}
	and \cref{ass:alg-assump} is satisfied.
	The result of the corollary now follows from \cref{thm:main-conv} after inserting the expressions for $\ell$ and $L_k$ into the convergence criterion $0 < \epsilon \leq 1 - L_k - L_{k-1} - \tau\frac{\ell}{2}$.
\end{proof}

\subsubsection*{Related Algorithms}
From \cref{alg:primdual-tri}, when $E=0$ and $\lambda_{k} = 2$ for all $k\in\{-1,0,\dots\}$, we obtain an instance of the V\~u\--Condat algorithm \cite{vuSplittingAlgorithmDual2013,condatPrimalDualSplitting2013}.
If $F=0$ as well, we get the method of Chambolle\--Pock \cite{chambolleFirstOrderPrimalDualAlgorithm2011}.
This is not surprising since both of these methods are special cases of ordinary forward-backward splitting and the kernel $M_k$, see \cref{eq:primdual-tri-params}, is linear, self-adjoint, and can be made strongly positive when $E=0$ and $\lambda_k=2$.
Furthermore, we have that $\gamma_kM_k - S = 0$, which implies that the momentum-correction term is zero and that \cref{alg:nofob_mom} has reduced to the ordinary forward-backward method.
Both when $F\neq0$ and when $F=0$, \cref{cor:primdual-tri-conv} regains the convergence criteria of V\~u\--Condat and Chambolle\--Pock respectively.

When $E=0$, \cref{alg:primdual-tri} shares similarities with the asymmetric-kernel primal-dual method of Latafat and Patrinos \cite[Algorithm 3]{latafatAsymmetricForwardBackward2017}.
They use the same resolvent kernel, but \cite{giselssonNonlinearForwardBackwardSplitting2021} showed that the Latafat\--Patrinos algorithm is a special case of nonlinear forward-backward splitting with projection correction instead of momentum correction.
As discussed in \cref{sec:prob-alg} when comparing momentum and projection corrections, the main benefit of \cref{alg:primdual-tri} is that the momentum correction generally yields cheaper iterations.
In \cref{alg:primdual-tri}, the linear composition term $V$ and its adjoint $V^*$ only need to be evaluated once each, while they need to be evaluated twice each for the Latafat\--Patrinos method.

We can also relate \cref{alg:primdual-tri} to projective splitting methods \cite{ecksteinGeneralProjectiveSplitting2009,combettesAsynchronousBlockiterativePrimaldual2018}.
It has been shown in \cite{giselssonNonlinearForwardBackwardSplitting2021arxiv,buiWarpedResolventSetValued2021} that these methods are nonlinear forward-backward method with projection correction.
In fact, the synchronous projective splitting considered in \cite{giselssonNonlinearForwardBackwardSplitting2021arxiv}
is using the same kernel as in \cref{alg:primdual-tri} with $E=0$ and $\lambda_k = 0$.
We can therefore think of \cref{alg:primdual-tri} with $E=F=0$ and $\lambda_{k} = 0$ for all $k\in\{-1,0,\dots\}$ as a projective splitting method with momentum correction instead of a projection correction.
The benefit of projective splitting methods compared to Chambolle\--Pock-like primal-dual methods is that the primal and dual updates do not depend on each other and can therefore be performed in parallel.
The same holds for \cref{alg:primdual-tri} since the correction $v_{k+1}$ does not depend on $y_{k+1}$ when $\lambda_k = 0$.
The reason for this becomes evident when examining the backward step $(M_k+A)^{-1} = (\widehat{M}_k+\widehat{A})^{-1}$ since both $\widehat{M}_k$ and $\widehat{A}$ are block-diagonal when $\lambda_k = 0$, see \cref{eq:primdual-operators,eq:primdual-tri-params}.

\subsubsection*{Forward-Half-Reflected-Douglas--Rachford}
There is a connection between primal\-/dual methods and Douglas\--Rachford splitting \cite{chambolleFirstOrderPrimalDualAlgorithm2011,oconnorEquivalencePrimaldualHybrid2020,briceno2021split,briceno2023primal}, and this connection also exists for our first primal-dual method, \cref{alg:primdual-tri}.
Whenever $V=\Id$ and $F=0$, choosing $\lambda_{k} = 2$ for all $k\in\{-1,0,\dots\}$, $\sigma = \varsigma^{-1}$ for some $\varsigma > 0$ and using Moreau's identity in the dual update of \cref{alg:primdual-tri}, the forward-reflected-Douglas\--Rachford (FRDR) method in \cite{ryuFindingForwardDouglasRachfordForward2020} is obtained.
Since we can allow for $F\neq 0$, we can analogously construct a forward-half-reflected-Douglas\--Rachford method, presented in \cref{alg:fhrdr}, for solving \cref{eq:primdualprob-prim}.
\begin{algorithm}[H]
	\caption{Forward-Half-Reflected-Douglas--Rachford}
	\label{alg:fhrdr}
	Consider problem \cref{eq:primdualprob-prim} with $V=\Id$.
	With $y_0, y_{-1} \in \Hsec$ and $z_0\in\Hdual$, for all $k\in\N$ iteratively perform
	\begin{align*}
		y_{k+1} &= (\Id + \tau B)^{-1} (y_k - \tau z_k - \tau ( 2 Ey_k - Ey_{k-1}) - \tau Fy_k), \\
		\hat{y}_{k+1} &= (\Id + \varsigma D)^{-1}(\varsigma z_k + 2y_{k+1}-y_k), \\
		z_{k+1} &= z_k + \varsigma^{-1}(2y_{k+1}-y_k - \hat{y}_{k+1}),
	\end{align*}
	where $\tau,\sigma > 0$.
\end{algorithm}

\Cref{alg:fhrdr} converges as per the following result.
\begin{cor}\label{cor:fhrdr-conv}
	Let $V=\Id$ and let \cref{ass:fhrb-prob-ass} hold.
	Consider problem \cref{eq:primdualprob-prim} and \cref{alg:fhrdr}.
	If the step-sizes satisfy
	\begin{align*}
		\tau(\varsigma^{-1} + 2\delta + \tfrac{1}{2}\beta) < 1,
	\end{align*}
	then $y_k \rightharpoonup y^\star$ and $z_k \rightharpoonup z^\star$ where $y^\star$ is a solution to \cref{eq:primdualprob-prim} and $(y^\star,z^\star)$ is a solution to \cref{eq:primdualprob}.
	\begin{proof}
		Follows directly from \cref{cor:primdual-tri-conv} with $V=\Id$ and $\lambda_{k-1}=2$ for all $k\in\N$.
	\end{proof}
\end{cor}
These convergence conditions match those of \cite{ryuFindingForwardDouglasRachfordForward2020} when $F=0$.

When $E=F=0$, the standard Douglas\--Rachford is retrieved from \cref{alg:fhrdr} if the step-sizes $\tau = \varsigma$ are chosen and the variable change $y_k - \tau z_k\to z_k$ is made.
However, this step-size choice makes the step-size condition of \cref{cor:fhrdr-conv} impossible to satisfy.
The reason for this is that the scaling $S$ of the underlying nonlinear forward-backward method becomes singular, which violates \cref{ass:prob-assump}.
Dealing with this singularity is possible if it is explicitly assumed that $E=F=0$, but this is beyond the scope of this article, where the positive definiteness of $S$ is assumed.

When $E=0$, \cref{alg:fhrdr} is applicable to the same class of problems as the David--Yin method in \cite{davisThreeOperatorSplittingScheme2017}. However, the algorithms are different, although they can both reduce to the Douglas--Rachford iterations when also $F=0$.

\subsection{Primal-Dual Method with Resolvent-Compensated Kernel}\label{sec:primdual-res}
Our second method for solving \cref{eq:primdualprob-prim} through the primal-dual problem \cref{eq:primdualprob} will make further use of the nonlinearity of the kernel by including resolvent evaluations in the kernel itself.
As in the previous case, we reformulate the primal-dual problem to our standard problem \cref{eq:prob} by defining $\mathcal{H}$, $A$, $C$, $\widehat{A}$, $\widehat{E}$, and $\widehat{V}$ as in \cref{eq:primdual-operators}.
The iterates of \cref{alg:nofob_mom} are decomposed as $x_k = (y_k,z_k)$ with $y_k\in\Hsec$ and $z_k\in\Hdual$ for all $k\in\N$.
The second primal-dual algorithm is then given by \cref{alg:nofob_mom} with the following design parameters:
\begin{equation}\label{eq:primdual-res-params}
	\begin{gathered}
		M_k = \underbrace{\begin{bmatrix} \tau^{-1}\Id - V^*\circ(\Id + \sigma D^{-1})^{-1}\circ T_{-z_k}\circ \sigma V & 0 \\ 0 & \sigma^{-1}\Id \end{bmatrix}}_{\widehat{M}_k} - \widehat{E}, \\
		S = \begin{bmatrix} \Id & 0 \\ 0 & \tau\sigma^{-1}\Id \end{bmatrix}
		\quad\text{and}\quad
		\gamma_k = \tau
	\end{gathered}
\end{equation}
where $\tau,\sigma > 0$ and $T_a \colon \Hdual \to \Hdual \colon z \mapsto z - a$ is the translation by $a \in \Hdual$.
Note that the current iterate $z_k$ is used in the construction of $M_k$ and that $S\in\mathcal{P}(\Hsec\times\Hdual)$ for all $\tau,\sigma > 0$.

With these design parameters, the correction operator becomes
\begin{equation}\label{eq:primdual-res-corrop}
	\gamma_kM_k - S
	=
	\tau \begin{bmatrix} -E -V^*\circ (\Id + \sigma D^{-1})^{-1}\circ T_{-z_k}\circ \sigma V  & 0 \\ 0 & 0 \end{bmatrix}.
\end{equation}
Inserting this and the other operators into the forward step,
\begin{equation*}
	\begin{aligned}
		(\hat{y}_k,\hat{z}_k) \coloneqq\;
		& M_k(y_k,z_k) - C(y_k,z_k) + \gamma_k^{-1}(\gamma_{k-1}M_{k-1}-S)(y_k,z_k) \\
		& - \gamma_k^{-1}(\gamma_{k-1}M_{k-1}-S)(y_{k-1},z_{k-1}),
	\end{aligned}
\end{equation*}
where $(\hat{y}_k,\hat{z}_k)\in\Hsec\times\Hdual$, yields
\begin{align*}
	\hat{y}_k
	&= \tau^{-1}y_k - (2Ey_k + Ey_{k-1}) - Fy_k \\
	&\quad - V^*(\Id + \sigma D^{-1})^{-1}(z_k + \sigma Vy_k) \\
	&\quad - V^*(\Id + \sigma D^{-1})^{-1}(z_{k-1} + \sigma Vy_k) \\
	&\quad + V^*(\Id + \sigma D^{-1})^{-1}(z_{k-1} + \sigma Vy_{k-1}), \\
	\hat{z}_k
	&= \sigma^{-1}z_k.
\end{align*}
To see that the backward step
\begin{align*}
	(M_k + A)^{-1} = (\widehat{M}_k - \widehat{E} + \widehat{A} + \widehat{E} + \widehat{V})^{-1} = (\widehat{M}_k + \widehat{A} + \widehat{V})^{-1},
\end{align*}
can be evaluated efficiently requires some extra attention.
The operator $\widehat{M}_k + \widehat{A} + \widehat{V}$ does not have the lower block-triangular structure as in the algorithm in \cref{sec:primdual-tri}.
We can therefore not evaluate its inverse using the same back substitution approach as before and computing it at a general point seems intractable.
However, $(\widehat{M}_k + \widehat{A} + \widehat{V})^{-1}$ is only evaluated at $(\hat{y}_k,\hat{z}_k)$ and the kernel has been specifically designed such that the backward step can be efficiently evaluated in this point.
First use
\begin{gather*}
	(y_{k+1},z_{k+1}) = (\widehat{M}_k + \widehat{A} + \widehat{V})^{-1}(\hat{y}_k,\hat{z}_k) \\
	\iff
	(\hat{y}_k,\hat{z}_k) \in (\widehat{M}_k + \widehat{A} + \widehat{V})(y_{k+1},z_{k+1}).
\end{gather*}
Writing out the inclusion problem explicitly yields
\begin{align*}
	\left\{
	\begin{alignedat}{3}
		\hat{y}_{k} &{}\in{} (\tau^{-1}\Id + B)y_{k+1} - V^*(\Id + \sigma D^{-1})^{-1}(z_k + \sigma Vy_{k+1}) + V^*z_{k+1}, \\
		\hat{z}_{k} &{}\in{} -Vy_{k+1} + (\sigma^{-1}\Id + D^{-1})z_{k+1}.
	\end{alignedat}
	\right.
\end{align*}
Using that $z_k = \sigma \hat{z}_k$ in the first row and solving for $z_{k+1}$ in the second row results in
\begin{align*}
	\left\{
	\begin{alignedat}{3}
		\hat{y}_{k} &{}\in{} (\tau^{-1}\Id + B)y_{k+1} - V^*(\Id + \sigma D^{-1})^{-1}(\sigma\hat{z}_k + \sigma Vy_{k+1}) + V^*z_{k+1}, \\
		z_{k+1} &{}={} (\Id + \sigma D^{-1})^{-1}(\sigma\hat{z}_k + \sigma Vy_{k+1}).
	\end{alignedat}
	\right.
\end{align*}
Inserting the second row into the first and solving for $y_{k+1}$ gives
\begin{align*}
	\left\{
	\begin{alignedat}{3}
		y_{k+1} &{}={} (\Id + \tau B)^{-1}(\tau\hat{y}_k), \\
		z_{k+1} &{}={} (\Id + \sigma D^{-1})^{-1}(\sigma\hat{z}_k + \sigma Vy_{k+1}).
	\end{alignedat}
	\right.
\end{align*}
Finally, inserting the expressions for $\hat{y}_k$ and $\hat{z}_k$ gives us the following algorithm.
\begin{algorithm}[H]
	\caption{Primal-Dual Method with Resolvent Corrected Kernel}
	\label{alg:primdual-res}
	Consider problem \cref{eq:primdualprob-prim}.
	With $y_0, y_{-1} \in \Hsec$ and $z_0,\nu_0\in\Hdual$, for all $k\in\N$ iteratively perform
	\begin{align*}
		\nu_{k+1} &= (\Id + \sigma D^{-1})^{-1}(z_k + \sigma Vy_k) \\
		y_{k+1} &= (\Id + \tau B)^{-1}( y_k - \tau V^*(z_k + \nu_{k+1} - \nu_{k}) - \tau (2Ey_k - Ey_{k-1}) - \tau Fy_k) \\
		z_{k+1} &= (\Id + \sigma D^{-1})^{-1}(z_k + \sigma Vy_{k+1})
	\end{align*}
	where $\tau,\sigma > 0$.
\end{algorithm}

We see that, compared to our other primal-dual method \cref{alg:primdual-tri}, we require one extra evaluation of the resolvent of $D^{-1}$ each iteration.
Apart from that, \cref{alg:primdual-res}, also only requires one evaluation of $(\Id + \tau B)^{-1}$, $V$ and $V^*$, given that $Vy_{k+1}$ is stored for the next iteration.
Still, the resulting per-iteration computational cost is higher compared to \cref{alg:primdual-tri} and most other primal-dual methods.
Exactly how much more expensive this method is will depend on the problem being solved and in some cases it is negligible.
The main reason for presenting \cref{alg:primdual-res}, apart from its novelty, is to further demonstrate the flexibility of the nonlinear kernel framework.

\begin{cor}\label{cor:primdual-res-conv}
	Let \cref{ass:primdualprob-ass} hold and consider problem \cref{eq:primdualprob-prim} and \cref{alg:primdual-res}.
	If the step-sizes satisfy
	\begin{align*}
		2\tau\sigma\norm{V}^2 + \tau(2\delta + \tfrac{\beta}{2}) < 1,
	\end{align*}
	then $y_k \rightharpoonup y^\star$ and $z_k \rightharpoonup z^\star$ where $y^\star$ is a solution to \cref{eq:primdualprob-prim} and $(y^\star,z^\star)$ is a solution to \cref{eq:primdualprob}.
	\begin{proof}
		Due to the structures of $S$ and $C$ we can conclude that $C$ is $\beta^{-1}$-cocoercive w.r.t.\ $S$ since
		\begin{align*}
			\norm{C(y,z)-C(y',z')}_{S^{-1}}^2
			&= \norm{Fy-Fy'}^2 \\
			&\leq \beta\inprod{Fy-Fy'}{y-y'} \\
			&= \beta\inprod{C(y,z)-C(y',z')}{(y,z)-(y',z')}
		\end{align*}
		for all $(y,z)\in\Hsec\times\Hdual$.
		We have previously established that $A$ is maximally monotone and, since we assume a solution exists, \cref{ass:prob-assump} holds.

		For \cref{ass:alg-assump}, we first note that $\gamma_k = \tau > 0$ and, hence, that the first assumption is satisfied.
		For the Lipschitz continuity of $\gamma_kM_k-S$ we recall the definition of the operator in \cref{eq:primdual-res-corrop}.
		The operator $E$ is, by assumption, $\delta$-Lipschitz continuous, and $(\Id + \sigma D^{-1})^{-1}\circ T_-{z_k}$ is $1$-Lipschitz since both the resolvent and translation are $1$-Lipschitz continuous.
		The operator $-\tau(E + V^*\circ (\Id + \sigma D^{-1})^{-1}\circ T_-{z_k} \circ \sigma V )$ is therefore $(\tau\delta + \tau\sigma\norm{V}^2)$-Lipschitz continuous for all $k\in\N$.
		Since
		\begin{align*}
			&\norm{(\gamma_kM_k-S)(y,z)-(\gamma_kM_k-S)(y',z')}_{S^{-1}}^2 \\
			&\qquad=
			\begin{aligned}[t]
				\| &\tau(E + V^*(\Id + \sigma D^{-1})^{-1}(z_k + \sigma V))y \\
				&- \tau(E + V^*(\Id + \sigma D^{-1})^{-1}(z_k + \sigma V) )y'\|^2
			\end{aligned} \\
			&\qquad\leq
			(\tau\delta+\tau\sigma\norm{V}^2)^2\norm{y-y'}^2 \\
			&\qquad\leq
			(\tau\delta+\tau\sigma\norm{V}^2)^2\norm{(y,z)-(y',z')}_S^2
		\end{align*}
		for all $(y,z)\in\Hsec\times\Hdual$, $\gamma_kM_k-S$ is $(\tau\delta + \tau\sigma\norm{V}^2)$-Lipschitz continuous w.r.t.\ $S$ for all $k\in\N$.
		The result now follows from \cref{thm:main-conv}.
	\end{proof}
\end{cor}

\begin{rem}
	As stated in \cref{rem:add-mom-tract}, the approach for adding momentum presented in \cref{sec:add-momentum} and \cref{alg:nofob_mom_add_mom} does not yield a tractable algorithm when applied to \cref{alg:primdual-res}.
	The kernel of \cref{alg:primdual-res} was designed in such a way that the backward step is only cheaply computed at the point given by the forward step and it is therefore not straightforward to apply the latter to the forward step with momentum.
	However, this is easily fixed.
	We regain computability of the backward step if we add $\theta(z_k-z_{k-1})$ according to
	\begin{align*}
		M_k = \begin{bmatrix} \tau^{-1}\Id - V^*\circ (\Id + \sigma D^{-1})^{-1}\circ T_{-z_k - \theta(z_k - z_{k-1})}\circ \sigma V & 0 \\ 0 & \sigma^{-1}\Id \end{bmatrix} - \widehat{E}
	\end{align*}
	and use this kernel in \cref{alg:nofob_mom_add_mom} instead.
	Since this operator only differs from the one in \cref{eq:primdual-res-params} by a translation, it does not modify any Lipschitz constants, and the convergence can be proved using the same approach as in \cref{cor:add-mom-conv}.
\end{rem}

\section{Conclusion}

We have presented a forward-backward method with a nonlinear resolvent and a novel momentum correction.
The design freedom of the nonlinear resolvent allows us to interpret numerous methods as special cases of this forward-backward method.
Existing special cases include the forward-(half)-reflected-backward method, the forward-reflected-Douglas\--Rachford method and the primal-dual methods of V\~u\--Condat and Chambolle\--Pock.
New algorithms include momentum versions of the previously mentioned algorithms and two new four-operator primal-dual splitting methods.
Our convergence conditions either regain or improve on the already known conditions for the existing methods, establishing parity of our more general analysis with the more specialized approaches.
We believe that this parity of analysis and the great amount of freedom in the parameter choices of our algorithm can prove useful for the understanding of existing algorithms and the development of new ones.

\paragraph{Acknowledgments}
All authors have been supported by ELLIIT: Excellence Center at Linköping-Lund in Information Technology.
The first and last authors have also been provided founding by the Swedish Research Council.
The Wallenberg AI, Autonomous Systems and Software Program (WASP) have supported the work of the second and last author.

\bibliography{references}

\end{document}